\documentclass[12pt,reqno]{amsart}

\usepackage{enumitem} 
\setlist{leftmargin=26pt}

\usepackage{amsmath,amsfonts,amssymb,amscd,amsthm,calc}
\usepackage[english]{babel}
\usepackage{yfonts,color}

\usepackage{tikz}

\newtheorem{theorem}{Theorem}[section]
\newtheorem{proposition}[theorem]{Proposition}
\newtheorem{lemma}[theorem]{Lemma}

\theoremstyle{definition}
\newtheorem{definition}[theorem]{Definition}

\newcommand{\darrow}{\!\downarrow}
\newcommand{\uarrow}{\!\uparrow}
\newcommand{\la}{\langle}
\newcommand{\ra}{\rangle}
\newcommand{\concat}{\mbox{}^\smallfrown}

\newcommand{\cmp}[1]{\overline #1}
\newcommand{\bigset}[1]{\big\{ #1 \big\}}

\renewcommand{\leq}{\leqslant}
\renewcommand{\geq}{\geqslant}

\newcommand{\fa}{\forall}
\newcommand{\ex}{\exists}

\newcommand{\A}{\mathcal{A}}
\newcommand{\B}{\mathcal{B}}
\newcommand{\E}{\mathcal{E}}
\newcommand{\G}{\mathcal{G}}
\newcommand{\K}{\mathcal{K}}

\newcommand{\dom}{\mathrm{dom}}

\newcommand{\eff}{\mathrm{eff}}

\newcommand{\emb}{\hookrightarrow}

\DeclareMathOperator{\graph}{\mathrm{graph}}
\newcommand{\pa}{\mathrm{PA}}
\newcommand{\Sc}{\mathcal{S}}

\AtBeginDocument{%
   \def\MR#1{}
}

\begin{document}

\title[Nonembeddings of Combinatory Algebras]
{Nonembeddings of Combinatory Algebras}

\author[P. Lutz]{Patrick Lutz}
\address[Patrick Lutz]{
University of Michigan\\
Department of Mathematics\\
2074 East Hall\\
530 Church Street\\
Ann Arbor, MI 48109-1043\\
U.S.A.
}
\email{pglutz@umich.edu}
\urladdr{https://websites.umich.edu/~pglutz/}

\author[P. Shafer]{Paul Shafer}
\address[Paul Shafer]{School of Mathematics\\
University of Leeds\\
Leeds\\
LS2 9JT\\
United Kingdom}
\email{p.e.shafer@leeds.ac.uk}
\urladdr{https://peshafer.github.io}

\author[S. A. Terwijn]{Sebastiaan A. Terwijn}
\address[Sebastiaan A. Terwijn]{Radboud University Nijmegen\\
Department of Mathematics\\
P.O. Box 9010, 6500 GL Nijmegen, the Netherlands.} 
\email{terwijn@math.ru.nl}
\urladdr{https://www.math.ru.nl/~terwijn/}

\begin{abstract} 
In the theory of combinatorial algebras, there is a sequence of 
embeddings between Kleene's second model, van Oosten's model, and 
Scott's graph model. We prove that none of these embeddings can 
be reversed.  
We also prove nonembedding results for the effective versions of 
these models, and in addition we discuss relativized embeddings. 
This answers several questions from the literature.  
\end{abstract}



\date{\today}

\maketitle

\section{Introduction}

We study embeddings and nonembeddings of partial combinatory algebras (pcas). 
Combinatory algebra is a formalism from logic that is closely related to 
lambda calculus. Both formalisms were conceived as theories of computation. 
In both formalisms, the application operator is total, since the application 
of one term to another always yields another term, even if this is 
computationally meaningless. For a theory of computation, this is not always 
desirable. It was long known that a treatment of combinatory algebra
with a partial application operator is possible (as is clear from the discussion 
about completions of pcas in the early 1970s), but the first written account by 
Feferman \cite{Feferman} came only in 1975. For more about the relation between 
combinatory algebra and lambda calculus we refer the reader to 
Barendregt~\cite{Barendregt}.

As might be expected, there are numerous connections between combinatory algebra
and computability theory, yet another theory that emerged in the 1930s. 
Still, after the initial period where Kleene established fundamental results 
such as the equivalence of computability and $\lambda$-definability, 
the two areas developed largely independently. It took until Scott~\cite{Scott} 
before further significant connections were made, when he introduced models of 
the lambda calculus directly based on notions from computability theory. 
As Scott says in his paper (p183)
\begin{quote}
$\lambda$-calculus and recursion theory make a good combination.
\end{quote}
This statement serves as a motto of his paper, 
and it was borne out by much of the subsequent work. 
Indeed, there remain a lot of relations between the two areas to be 
explored, and the current paper may be seen as a contribution to this area. 
We use the notions and techniques from computability theory to say something 
about the relations between some of the standard models in 
combinatory algebra. 
We start with a quick review of the basic definitions. 

A {\em partial applicative system (pas)\/} is a set $A$ with a partial 
binary application operator $\cdot$. For any two elements $a,b\in A$, 
we have an application $a\cdot b$, that may or may not be defined. 
If it is defined we write $a\cdot b\darrow$. We also often omit the 
dot and simply write $ab$. We can consider all the terms built from 
elements of $A$, variables, and application. 
An applicative system is called a 
{\em partial combinatory algebra (pca)\/} if for every term 
$t(x_1,\ldots,x_n,x)$, $n\geq 0$, with free variables among
$x_1,\ldots,x_n,x$, there exists $b\in \A$ such that
for all $a_1,\ldots,a_n,a\in \A$,
\begin{enumerate}[label={\rm (\roman*)}] 

\item $ba_1\cdots a_n\darrow$,

\item $ba_1\cdots a_n a \simeq t(a_1,\ldots,a_n,a)$.

\end{enumerate}
Here $\simeq$ denotes the Kleene equality, meaning either both sides 
are defined and equal, or both are undefined.  
A combinatory algebra is simply a pca for which the application 
operator is total. 

The defining property of pcas (called combinatory completeness) 
is precisely characterized by the existence of special combinators 
$k$ and $s$, just as in classical combinatory algebra. 
The following theorem is implicit in Feferman~\cite{Feferman}.

\begin{theorem} \label{Feferman}
A pas $\A$ is a pca if and only if it has elements $k$ and $s$
with the following properties for all $a,b,c\in\A$:
\begin{itemize}

\item $ka\darrow$ and $kab = a$,

\item $sab\darrow$ and $sabc \simeq ac(bc)$.

\end{itemize}
\end{theorem}

Note that a pca with one element, and $k=s$, satisfies the definition 
of a pca. van Oosten \cite{vanOosten} calls this pca trivial. From now 
on we will only consider nontrivial pcas. Beeson~\cite[p100]{Beeson} 
includes nontriviality in the definition by requiring $k\neq s$. 
A nontrivial pca is always infinite, and can represent the natural numbers 
and all partial computable functions on the natural numbers. 

Since Theorem~\ref{Feferman} characterizes pcas, it is sometimes taken 
as a definition instead of a theorem. However, we would like to stress 
that $k$ and $s$, though they always exist,  
need not be considered as part of the signature of a pca. 
The reasons for this will become clear below when we talk about 
embeddings.

There is a large variety of examples of pcas in the literature. 
For example the lambda calculus itself is an example (where application is total). 
Further examples may be found in the monographs by 
Beeson~\cite{Beeson}, Odifreddi~\cite{Odifreddi}, 
van Oosten~\cite{vanOosten}, and Longley and Normann~\cite{LongleyNormann}.
Apart from their use in theoretical computer science, extensive use of pcas is 
being made in constructive mathematics, 
see Troelstra and van Dalen~\cite{TroelstravanDalenII}.
In particular, pcas may serve as a basis for models of constructive set theory, 
as for example in Rathjen~\cite{Rathjen} and 
Frittaion and Rathjen \cite{FrittaionRathjen}.
Here we only list the examples of pcas that play a role in this paper. 

The most important of all examples is Kleene's first model $\K_1$, 
which has as its elements the natural numbers, and application defined by  
\[
n \cdot m = \Phi_n(m),
\]
where $\Phi_n$ is the $n$\textsuperscript{th}
partial computable function. 
This is the setting of classical computability theory, and it 
underlines the need to consider partial applications. 
We can also relativize this to an arbitrary oracle $X$, thus obtaining 
the relativized pca $\K_1^X$. 

Kleene's second model $\K_2$, 
introduced in Kleene and Vesley~\cite{KleeneVesley}, 
is a pca defined on Baire space $\omega^\omega$.
Application $f \cdot g$ in this model can be informally 
described as applying the continuous functional with code $f$ 
to the real $g$. The original coding of $\K_2$ is somewhat contrived, 
but it is essentially equivalent to 
\begin{equation}\label{def:K2}
f\cdot g = \Phi_{f(0)}^{f\oplus g}, 
\end{equation}
where the application 
is understood to be defined if the RHS is total. 
This coding, used in Shafer and Terwijn \cite{ShaferTerwijn}, 
is considerably easier to work with. 
See the appendix of \cite{GolovTerwijn} for a proof (and precise 
statement) of the equivalence with the original coding. 

We can relativize the definition of $\K_2$ to an arbitrary oracle $X$ 
by restricting the definition of $\K_2$ to  $X$-computable sequences. 
This gives a countable pca $\K_2^X$ for every~$X$. 
Note that, in contrast to $\K_1$, the pca $\K_2^X$ is actually smaller 
than $\K_2$, which is uncountable. 
Taking $X$ computable gives the pca $\K_2^\mathrm{eff}$, 
consisting of all computable sequences, with application as in $\K_2$.

The van Oosten model $\B$, introduced in van Oosten~\cite{vanOosten1999}, 
is a variant of $\K_2$ that is obtained by extending the domain to include 
partial functions.
It can be described succinctly by taking the application in $\K_2$ as in 
\eqref{def:K2} to also include partial functions $f$ and $g$. 
By definition, an oracle computation $\Phi_e^{f\oplus g}(x)$ for partial functions 
is undefined as soon as it hits a query where the oracle is undefined. 
Like $\K_2$, the definition of $\B$ can be relativized to any oracle $X$
by restricting to the partial $X$-computable functions. 
This gives a pca $\B^X$ for any $X$, and when $X$ is computable we 
denote it by $\B^\eff$.
Note that unlike $\K_2$ and $\K_2^X$, both $\B$ and $\B^X$ are total 
combinatory algebras.

The graph model $\G$ was introduced in Scott~\cite{Scott},
as a continuation and improvement on his earlier work on models of the 
$\lambda$-calculus.\footnote{Scott 
also gives credit to the earlier authors Myhill and Shepherdson, Rogers, and 
Plotkin. He suggested the name graph model for $\G$ (p155) to prevent that it 
would be called Scott's model, so it is ironic that people now commonly refer 
to it as Scott's graph model.} 
$\G$ is a pca on the power set $\mathcal{P}(\omega)$, 
with application defined as 
\begin{equation}\label{appG}
X\cdot Y = \bigset{n\in\omega\mid \ex u (\la n,u\ra\in X \wedge D_u\subseteq Y)},
\end{equation}
where $D_u$ denotes the finite set with canonical code~$u$.
Note that $\G$ is in fact a total combinatory algebra. The connection with 
Rogers' notion of enumeration reducibility is that 
$Z\leq_e Y$ if and only if $Z = X\cdot Y$ for a c.e.\ set~$X$.

Define $\G^X$ as the smallest sub-pca of $\G$ containing $X$ and all c.e.\ sets. 
For $X$ computable, we denote $\G^X$ by $\E$, which is the usual notation for 
the class of c.e.\ sets. 
So $\E$ is the combinatory algebra of c.e.\ sets with application as in $\G$.

\begin{proposition}\label{prop:GX}
$\G^X = \{Z \in \mathcal{P}(\omega) \mid Z\leq_e X\}$.
\end{proposition}
\begin{proof}
The inclusion $\supseteq$ holds by the remarks about enumeration reducibility above. 
For $\subseteq$, we need that $X\cdot W_e$ can be written as $W_d\cdot X$ 
for some $d$. Intuitively it is clear that we can enumerate $X\cdot W_e$  
given an enumeration for $X$. We give no formal proof, as this also 
follows from the results in Scott \cite{Scott} quoted below. 
\end{proof}

Note that $\G^X$ is different from the $X$-c.e.\ sets. For example, when $X$ is c.e.\ 
then $\G^X = \E$, but the $X$-c.e.\ sets are a larger class when $X$ is not computable. 


\begin{definition} \label{def:fingen}
Call a pca $\A$ {\em finitely generated\/} if there is a 
finite set $C\subseteq \A$ such that every element of $\A$ can be 
written as a closed term composed of elements of~$C$.
\end{definition}

For the record, we list the following examples, most of which are known:

\begin{theorem} \label{thm:fingen}
The following pcas are all finitely generated:
\begin{enumerate}[label={\rm (\roman*)}]

\item $\K_1$, $\K_2^\eff$, $\B^\eff$, $\E$,

\item $\K_1^X$, $\K_2^X$, $\B^X$, $\G^X$ for any $X$, 

\item term algebras such as CL and $\lambda$-calculus. 

\end{enumerate}
\end{theorem}
\begin{proof}
Some of these examples, such as $\K_1^X$ and $\B^X$,  
have been used implicitly in \cite{GolovTerwijn}.
Note that the examples in (i) are special cases of those in~(ii).

$\K_1$ is finitely generated: Consider the code $e$ of a p.c.\ function 
such that $e\cdot e = 0$ and $e\cdot 0 = \text{successor function}$. 
Then the terms just containing $e$ yield all $n\in\omega$, hence $e$ 
by itself generates~$\K_1$. 
The same argument shows that $\K_1^X$ is also generated by one element. 

That $\B^X$ is finitely generated was discussed in footnote~2 in 
\cite{GolovTerwijn}. 
Namely we can define the Church numerals using only the basic combinators 
$s$ and $k$, and using these we can define the functions $h_n(x)=n$ 
for every~$n$. Then we choose $j\in\B^X$ such that 
\begin{equation} \label{BX}
\B^X \models j\cdot h_n = \Phi_n^X
\end{equation} 
for every~$n$.

That $\K_2^X$ is finitely generated follows from the argument for $\B^X$
just given. 
(Having the embedding $\K_2^X \emb \B^X$ by itself is not enough for 
this.)\footnote{
Embedding $\K_2^X$ into $\B^X$ (which is finitely generated) does not 
give that $\K_2^X$ is finitely generated, since a non-finitely generated pca 
may embed into a finitely generated one. 
For example, consider an infinite increasing chain $X_i$ of Turing degrees. 
Then the union of $\G^{X_i}$ is not finitely generated, but by 
Scott \cite{Scott} every countably generated sub-pca of $\G$ is included 
in a finitely generated one, namely we can take the join of the generators.}
Namely, the elements $j$ and $h_n$ in \eqref{BX} are in fact total, 
hence elements of $\K_2^X$, and since application in $\K_2^X$ is 
the same as in $\B^X$, for {\em total\/} $\Phi_n^X$ we also have 
$\K_2^X \models j\cdot h_n = \Phi_n^X$.

$\G^X$ is finitely generated: Scott~\cite[Theorem 3.5]{Scott} showed that 
the finitely generated sub-pcas of $\G$ are precisely those of the form $\G^X$.
(He also showed that $\G^X$ can be generated with only one generator.)
This also gives another proof of Proposition~\ref{prop:GX}.

That CL and the $\lambda$-calculus are finitely generated is well-known, 
cf.\ Barendregt~\cite{Barendregt}.
\end{proof}

Since every finitely generated pca is countable, we obviously 
have that uncountable pcas such as $\K_2$, $\B$, and $\G$ are 
not finitely generated. 
However, there are also examples of countable 
pcas that are not finitely generated.\footnote{An example of this is the 
countable product $\K_1^\omega$, with elements the eventually constant 
sequences, and pointwise application as in $\K_1$.}


We now discuss embeddings of pcas. 
To distinguish applications in different pcas, we write 
$\A\models ab\darrow$ if this application is defined in $\A$.

\begin{definition} \label{def:embedding}
For given pcas $\A$ and $\B$, an injection 
$f: \A \to \B$ is a \emph{weak embedding} if for all $a, b \in \A$, 
\begin{equation}\label{emb}
\A\models ab\darrow \; \Longrightarrow \;
\B\models f(a)f(b)\darrow \, = f(ab).
\end{equation}
If $\A$ embeds into $\B$ in this way we write $\A\hookrightarrow \B$.
If in addition to \eqref{emb}, for a specific choice of combinators 
$k$ and $s$ of $\A$, $f(k)$ and $f(s)$ serve as combinators for $\B$, 
we call $f$ a {\em strong embedding}. 
\end{definition}

Since by Theorem~\ref{Feferman} the combinators $k$ and $s$ always 
exist in any pca, it may seem insignificant if we consider them to 
be part of the signature or not, and consequently the difference 
between weak and strong embeddings might seem equally small. 
However, it does make a big difference. This is apparent for example 
in the study of completions of pcas (i.e.\ embeddings of pcas into 
total combinatory algebras). 
If we do not consider $k$ and $s$ to be part of the signature then 
(by a result of Engeler~\cite{Engeler}) every pca has a completion, but if we do 
then (by a result of Klop~\cite{Klop}) this is not the case. 
For more about completions see Terwijn~\cite{Terwijn2025}.
Both notions of embedding have been studies in the literature, 
sometimes by the same author, for example Bethke~\cite{Bethke}.
Consequences of embeddings of pcas (both weak and strong) 
for the resulting realizability models are discussed in Swan \cite{Swan}.
In addition to this, there is also an even weaker notion called 
applicative morphism, introduced in Longley~\cite{Longley}, that 
is useful in categorical contexts. 
In this paper we will focus on weak embeddings. 
As we will be proving nonembedding results, these will imply 
automatically the nonembedding results for strong embeddings.

Our notation is mostly standard.  
For computability theory we follow Odifreddi~\cite{Odifreddi}
and for partial combinatory algebra van Oosten~\cite{vanOosten}.
For unexplained notation we refer to these monographs.
In particular $\Phi^X_e$, $e\in\omega$, denotes the standard numbering of 
the partial computable functions relative to oracle $X$.

\section{Nonembeddings}

We consider weak embeddings of pcas. We have the following sequence of embeddings:
\begin{equation} \label{embs}
\K_1 \emb \K_2 \emb \B \emb \G.
\end{equation}
The first is by Shafer and Terwijn \cite{ShaferTerwijn}, 
the second is by inclusion, and 
the third is by Golov and Terwijn \cite{GolovTerwijn}.
Of course we have that $\K_2 \not\emb \K_1$, since $\K_1$ is countable 
and $\K_2$ is not, but it was open until now whether any of the other 
embeddings could be reversed. (This was posed as an open question in 
Golov and Terwijn \cite[Question 7.4]{GolovTerwijn}.) 
Below we show that $\B \not\emb \K_2$ and that $\G \not\emb \B$. 

Since $\G$ is related to enumeration reducibility, and $\K_2$ to Turing reducibility, 
the question whether $\G \not\emb \K_2$ seems related to embedding the 
enumeration degrees into the Turing degrees. 
It is known by a general result of Abraham and Shore \cite{AbrahamShore} 
(cf.\ Odifreddi \cite[p529]{Odifreddi})
that it is consistent (assuming CH) that the enumeration degrees are isomorphic 
to an initial segment of the Turing degrees (as an upper semilattice).
Now any ideal $I$ in the Turing degrees $\mathcal{D}_T$ gives 
a pca with the application of $\K_2$. 
(E.g.\ the relativized versions $\K_2^X$ of $\K_2$ are of this form.) 
In particular this is true for the embedded version of the 
enumeration degrees $\mathcal{D}_e$. 
However, the application $\cdot_{\K_2}$ may not match the application 
in $\G$, so this does not necessarily give an embedding of $\G$ into $\K_2$, 
and indeed our results shows that such an embedding is impossible.

\begin{theorem}\label{thm:BnotinK2}
$\B \not\emb \K_2$.
\end{theorem}
\begin{proof} 
Suppose for a contradiction that $F:\B\emb\K_2$ is a weak embedding. 
So we have for all $f,g\in\B$, 
$$
\B\models f\cdot g = h \Longrightarrow 
\K_2\models F(f)\cdot F(g) \darrow = F(h).
$$
Note that application in $\B$ is total, so $f\cdot g$ is always defined. 
Let $u\in \B$ be totally undefined, i.e.\ $u(x)\uarrow$ for all $x$, 
and let $h\in\B$ be an element different from~$u$.
We can define $a\in\B$ such that for all {\em total\/} $f,g$, 
$$
\B\models a\cdot f\cdot g = 
\begin{cases}
u & \text{if $f=g$} \\
h & \text{if $f\neq g$.}
\end{cases}
$$ 
The result of $a\cdot f\cdot g$ for $f,g$ nontotal is immaterial.

Since $F$ is injective, we have $F(u)\neq F(h)$, and therefore
there are $n$ and $y_0 \neq y_1$ such that
$F(u)(n) = y_0$ and $F(h)(n)= y_1$.

For every total $f\in\B$ we have 
$$
\K_2\models (F(a)\cdot F(f) \cdot F(f))(n) = y_0.
$$ 
Hence for every total $f\in\B$ there exists a finite initial segment
$\sigma_f\sqsubset F(f)$ such that 
$$
\K_2\models (F(a)\cdot \sigma_f\cdot \sigma_f)(n) = y_0.
$$
(By this we mean that for any $g_0, g_1 \sqsupset \sigma_f$, 
$\K_2\models (F(a)\cdot g_0 \cdot g_1)(n) = y_0$.)
By counting there exist total $f,g\in\B$ such that $f\neq g$ and 
$\sigma_f = \sigma_g$.
Then we have 
$(F(a)\cdot \sigma_f\cdot \sigma_g)(n) = y_0$ (since $\sigma_f = \sigma_g$). 
But $f\neq g$, hence 
$F(a)\cdot F(f) \cdot F(g) = F(h)$, and $F(h)(n)= y_1$, a contradiction. 
\end{proof}

\begin{theorem}\label{thm:GnotinB}
$\G \not\emb \B$.
\end{theorem}
\begin{proof} 
Suppose for a contradiction that 
$F:\G\emb\B$ is a weak embedding, so that for all $X,Y,Z\in\G$,  
$$
\G\models X\cdot Y =Z \Longrightarrow 
\B\models F(X)\cdot F(Y) = F(Z).
$$
Note that both $\G$ and $\B$ are total, so that the applications 
are always defined. 

Call the embedding $F$ {\em monotone\/} if 
$$
\fa A\subseteq B \,\fa n \,( F(A)(n)\darrow \Longrightarrow F(B)(n)\darrow = F(A)(n)).
$$
In the rest of the proof we treat the two cases where $F$ is or is not monotone 
separately. 

{\em Case 1.\/} Suppose that $F$ is not monotone. 
This case of the proof uses an idea similar to that of Theorem~\ref{thm:BnotinK2}.
Since $F$ is not monotone there exist sets $A\subsetneq B$ and $n\in\omega$ such 
that $F(A)(n)\darrow \,\not\simeq\, F(B)(n)$. 
Fix such $A,B$ and $n$, and suppose $F(A)(n)\darrow = y$.

We can easily define an enumeration operator $C\in\G$ such that 
for all $X,Y\in \G$, 
\begin{align*}
X=Y &\Longrightarrow \G\models C \cdot (X\oplus\cmp{X}) \cdot (Y\oplus\cmp{Y}) = A, \\
X\neq Y &\Longrightarrow \G\models C \cdot (X\oplus\cmp{X}) \cdot (Y\oplus\cmp{Y}) = B.
\end{align*}
Since $F$ is an embedding, we have 
$$
X=Y \Longleftrightarrow 
\B \models (F(C) \cdot F(X\oplus\cmp{X}) \cdot F(Y\oplus\cmp{Y}))(n)\darrow = y.
$$
In particular we have 
$$
\B \models (F(C) \cdot F(X\oplus\cmp{X}) \cdot F(X\oplus\cmp{X}))(n)\darrow = y
$$
for every $X$.
Hence for each $X$ there is a finite initial segment\footnote{Since 
$F(X\oplus\cmp{X})$ is a partial function, the initial segment 
$\sigma_X$ is a finite partial function. 
For a partial function $f$, by a finite initial 
segment $\sigma\sqsubset f$ we mean a finite partial function $\sigma$ 
such that for every $n$, $\sigma(n)\darrow$ implies $f(n)\darrow = \sigma(n)$
and also $\sigma(n)\uarrow$ implies $f(n)\uarrow$.}
$\sigma_X \sqsubset F(X\oplus\cmp{X})$
such that for all $g_0,g_1\sqsupset \sigma_X$, 
$$
(F(C) \cdot g_0 \cdot g_1)(n)\darrow = y.
$$
By counting we see that there exist $X\neq Y$ such that $\sigma_X = \sigma_Y$. 
We then have 
$$
(F(C) \cdot F(X\oplus\cmp{X}) \cdot F(Y\oplus\cmp{Y}))(n)\darrow = y
$$
because $\sigma_X = \sigma_Y \sqsubset F(X\oplus\cmp{X}), F(Y\oplus\cmp{Y})$.
On the other hand, since $X\neq Y$ we have 
$$
(F(C) \cdot F(X\oplus\cmp{X}) \cdot F(Y\oplus\cmp{Y}))(n) = F(B)(n) \not\simeq y.
$$
Thus we have a contradiction.

{\em Case 2.\/} Suppose that $F$ is monotone.
Let $R\in\G$ be such that for all sets $X$ and $Y$, 
\begin{itemize}

\item if $0\notin X\cup Y$ then $R\cdot X\cdot Y = \{0\}$, 

\item if $0\in X\cup Y$ then $R\cdot X\cdot Y = \{0,1\}$.

\end{itemize}
It is easy to see that such an enumeration operator $R$ can be defined in $\G$ 
(or even in the effective version $\E$). 

Now let $X\subseteq Y$ be sets such that $0\notin X$ and $0\in Y$. 
Then $R\cdot X\cdot Y = R\cdot Y\cdot X = \{0,1\}$ and 
$R\cdot X\cdot X = \{0\}$.
Hence in $\B$ we have 
$F(R)\cdot F(X)\cdot F(Y) = F(R)\cdot F(Y)\cdot F(X)$, 
and by monotonicity (and injectivity) of $F$ there exists $n$ 
in the domain of $F(R)\cdot F(X)\cdot F(Y)$ that is not 
in the domain of $F(R)\cdot F(X)\cdot F(X)$.
Now consider the undefined computation 
$(F(R)\cdot F(X)\cdot F(X))(n)$.
This computation in $\B$ can be undefined in one of two ways: 
It can be undefined because of a query to an undefined spot in 
the oracle $F(R)$ or one of the copies of $F(X)$, or all the oracle 
queries in the computation are defined, but the computation still 
does not converge. 
We claim that at least one of the copies of $F(X)$ 
is not queried on any input where it is undefined. 
This is clear if the computation does not query any of the three components 
in an undefined spot, and if the 
computation queries either $F(R)$ or one of the copies $F(X)$ on an 
undefined input, then the computation is immediately undefined for this 
reason, hence the other copy $F(X)$ is never queried on an 
undefined input. 
Suppose the first copy of $F(X)$ is never queried on an undefined input. 
Then by monotonicity, we can substitute $F(Y)$ for $F(X)$ to obtain
$$
(F(R)\cdot F(X)\cdot F(X))(n) = (F(R)\cdot F(Y)\cdot F(X))(n), 
$$
and the latter computation is defined on $n$, contradicting that the 
first is undefined. 
If the second copy of $F(X)$ is never queried on an undefined input 
we reach a contradiction in the same way. 
\end{proof}

\section{Relativized embeddings}

For any $X$, the embeddings \eqref{embs} restrict to:
\[
  \K_1^X \emb \K_2^X \emb \B^X \emb \G^{X \oplus \cmp X}.
\]
In particular, for $X$ computable we have 
\[
  \K_1 \emb \K_2^\eff \emb \B^\eff \emb \E.
\]
By Golov and Terwijn \cite{GolovTerwijn} we know that 
\[
  \B^X \not\emb \K_2^X \not\emb \K_1^X
\]
for any $X$, 
and also that $\E\not\emb\K_2^\eff$.
Until now it was open whether $\E \not\emb \B^\eff$. 
We prove this in Theorem~\ref{thm:EnotinBeff} below.
Note that this is the effective analog of Theorem~\ref{thm:GnotinB}.
Also note that we do have embeddings $\B^X \emb \K_2$ for every $X$, 
since $\B^X$ is countable and every countable pca embeds into $\K_2$ 
(\cite[Corollary 6.2]{GolovTerwijn}). 
However, it is not the case that for $X\leq_T Y$ the embedding of $\B^Y$
extends that of $\B^X$, so the embeddings of $\B^X$ do not amalgamate to 
an embedding of $\B$ into $\K_2$, which by Theorem~\ref{thm:BnotinK2} is 
indeed impossible.

\begin{theorem}\label{thm:EnotinBeff}
$\E \not\emb \B^\eff$.
\end{theorem}
\begin{proof}
Suppose for a contradiction that 
$F:\E\emb\B^\eff$ is a weak embedding, so that for all $X,Y,Z\in\E$,  
\begin{equation}\label{emb2}
\E\models X\cdot Y =Z \Longrightarrow 
\B^\eff\models F(X)\cdot F(Y) = F(Z).
\end{equation}
Since $\E$ and $\B^\eff$ are total, all applications are always defined. 
Although this effective version of Theorem~\ref{thm:GnotinB} is different, 
if only for cardinality reasons (in the previous proof we 
used a counting argument), we make the same case distinction as before. 
So we call the embedding $F$ {\em monotone\/} if 
$$
\fa A\subseteq B \,\fa n \,( F(A)(n)\darrow \Longrightarrow F(B)(n)\darrow = F(A)(n)).
$$
In the rest of the proof we treat the two cases where $F$ is or is not monotone 
separately. 

{\em Case 1.\/} 
Suppose that $F$ is not monotone.
Then there exist c.e.\ sets $A,B\in \E$ with $A\subseteq B$ 
and $x\in\omega$ such that 
$F(A)(x)\darrow \,\not\simeq\, F(B)(x)$. 
Suppose $F(A)(x)\darrow = y$.

Now we can, using that $\E$ is finitely generated, 
show that $\emptyset'$ is decidable. 
Let $d$ be a computable function such that 
$$
W_{d(n)} = 
\begin{cases}
A & \text{if } n\notin\emptyset' \\
B & \text{if } n\in\emptyset'
\end{cases}
$$
for every $n$.
By Theorem~\ref{thm:fingen}, 
$\G^X$ is finitely generated for every $X$. 
In particular, $\E$ is finitely generated, say by generators $G_1,\ldots, G_k$. 
(In fact, Scott showed that one generator suffices.)
In fact, for every $n$ we can effectively compute a term $t_n(G_1,\ldots,G_k)$ 
composed entirely of generators of $\E$ such that $\E \models t_n(G_1,\ldots,G_k) = W_n$.
From this we can compute a code of $F(W_{d(n)})$ in $\B^\eff$ for every $n$. 
Namely, for every $n$ we have 
$\E \models t_{d(n)}(G_1,\ldots,G_k) = W_{d(n)}$.
By \eqref{emb2} we then have that the term 
$t_{d(n)}(F(G_1),\ldots, F(G_k))$ denotes $F(W_{d(n)})$ in $\B^\eff$.
So even if $F$ is noneffective, all we need are codes for the 
images of the generators $F(G_1),\ldots, F(G_k)$.
If $n\notin\emptyset'$ we have $F(W_{d(n)})(x)\darrow = y$, 
so we can decide whether $n\in\emptyset'$ by searching until 
either $n\in\emptyset'$ or $F(W_{d(n)})(x)\darrow = y$.

{\em Case 2.\/} 
Suppose that $F$ is monotone.
This case is in fact identical to the case in the 
previous proof that $\G\not\emb \B$ (Theorem~\ref{thm:GnotinB}).
We only need to observe that the enumeration operator 
$R$ defined  there, with the property that 
$R\cdot X\cdot Y = \{0\}$ if $0\notin X\cup Y$ and 
$R\cdot X\cdot Y = \{0,1\}$ if $0\in X\cup Y$, 
can indeed be chosen to be in $\E$. 
The rest of the argument is the same as before. 
\end{proof}

\begin{theorem}
$\G^{X \oplus \cmp X} \not\emb \B^X$ for any $X$. 
\end{theorem}
\begin{proof}
This is the relativization of $\E \not\emb \B^\eff$. 
It is straightforward to check that the proof of the latter relativizes.
In the first part of the proof, where $F$ is not monotone, we have to 
replace $\emptyset'$ by $X'$, and use that 
$\G^X$ is finitely generated by Theorem~\ref{thm:fingen}.
The second part of the proof, where $F$ is monotone, is purely combinatorial
and can be copied almost verbatim.
\end{proof}

\section{Embedding $\K_2^X$ into $\K_1^Y$}

Let $f: \A \to \B$ be an embedding from pca $\A$ into pca $\B$, and let $a, b \in \A$.  If 
$\A\models 
ab\darrow$, then $\B\models f(a)f(b)\darrow = f(ab)$.  However, if $\A\models 
ab\uarrow$, then it may be that either $\B\models f(a)f(b)\uarrow$ or $\B\models 
f(a)f(b)\darrow$. 
 If also $\B\models f(a)f(b)\uarrow$ whenever $\A\models ab\uarrow$, we say that the 
embedding $f$ \emph{preserves undefined}.

In this section, we consider the question of whether $\K_2^X$ embeds into $\K_1^Y$ for oracles 
$X$ and $Y$.  We find that the answer depends on whether the embedding is required to 
preserve undefined.  First, we show that $\K_2^X$ embeds into $\K_1^Y$ via an embedding 
that preserves undefined if and only if $X'' \leq_T Y$.  Second, we show that $\K_2^X$ embeds 
into $\K_1^{X'}$ via an embedding that does not preserve undefined.  This confirms the first part 
of~\cite[Conjecture~6.20]{GolovTerwijn}.  By~\cite[Theorem~8.4]{ShaferTerwijn} (see 
also~\cite[Section~4]{GolovTerwijn}), $K_1^X \emb K_1^Y$ if and only if $X \leq_T 
Y$.\footnote{\cite[Theorem~8.4]{ShaferTerwijn} 
considers embeddings that preserve undefined, but 
the same proof works for embeddings that do not necessarily preserve undefined.}  Therefore, if 
$X' \leq_T Y$, then $\K_2^X \emb \K_1^{X'} \emb \K_1^Y$, hence $\K_2^X \emb \K_1^Y$.  
Conversely, if $\K_2^X \emb \K_1^Y$, then $X' \leq_T Y$ by~\cite[Theorem~6.9]{GolovTerwijn}.  
Thus $\K_2^X \emb \K_1^Y$ if and only if $X' \leq_T Y$.

\begin{lemma} \label{lem:K2embK1presundef}
For every oracle $X$, $\K_2^X$ embeds into $\K_1^{X''}$ via an embedding that preserves 
undefined.
\end{lemma}

\begin{proof}
Using the recursion theorem, let $e$ be an index such that $\Phi_e^{X''}(a)$ produces another 
index $p = p(a) > 0$ for a machine $\Phi_{p(a)}^{X''}$ behaving as follows.
\begin{itemize}
\item $\Phi_{p(a)}^{X''}(0) = a$.
\item For $n > 0$, $\Phi_{p(a)}^{X''}(n)$:
\begin{itemize}
\item First runs $\Phi_n^{X''}(0)$.  If it halts, lets $b = \Phi_n^{X''}(0)$.
\item Then uses $X''$ to determine whether $(\Phi_a^X) \cdot (\Phi_b^X) \darrow$ (i.e., is total).
\item If $(\Phi_a^X) \cdot (\Phi_b^X) \darrow$, uses $X''$ to find the least $c$ such that 
$\Phi_c^X 
= (\Phi_a^X) \cdot (\Phi_b^X)$ and returns $\Phi_e^{X''}(c)$.
\item Here $\Phi_{p(a)}^{X''}(n)\uarrow$ if either $\Phi_n^{X''}(0)\uarrow$ or $(\Phi_a^X) \cdot 
(\Phi_b^X) \uarrow$.
\end{itemize}
\end{itemize}
For each $f \in \K_2^X$, let $a_f$ be its \emph{least representative}, i.e., the least number $a_f$ 
such that $\Phi_{a_f}^X = f$.  We show that $f \mapsto e a_f$ embeds $\K_2^X$ into 
$K_1^{X''}$ 
and preserves undefined.

Note that for any $a$, $ea0 = p(a) \cdot 0 = \Phi_{p(a)}^{X''}(0) = a$.  Thus if $f \neq g$, then 
$a_f \neq a_g$, so $e a_f \neq e a_g$ because $e a_f 0 = a_f \neq a_g = e a_g 0$.  Hence the 
mapping is injective.

Suppose that $fg \darrow = h$ in $\K_2^X$.  Then
\begin{align*}
(e a_f) \cdot (e a_g) = \Phi_{e a_f}^{X''}(e a_g) = \Phi_{p(a_f)}^{X''}(e a_g) = \Phi_e^{X''}(a_h) = e 
a_h.
\end{align*}
For the third equality, $e a_g 0 = a_g$ as explained above, and $(\Phi_{a_f}^X) \cdot 
(\Phi_{a_g}^X) 
= fg = h$ is total, so $a_h$ is the least $c$ such that $\Phi_c^X = (\Phi_{a_f}^X) \cdot 
(\Phi_{a_g}^X)$.

If instead $fg \uarrow$ in $\K_2^X$, then $(\Phi_{a_f}^X) \cdot (\Phi_{a_g}^X)\uarrow$, so $(e 
a_f) \cdot (e a_g) = \Phi_{e a_f}^{X''}(e a_g)\uarrow$ in $\K_1^{X''}$.
\end{proof}

\begin{lemma} \label{lem:presundefjj}
For oracles $X$ and $Y$, if $\K_2^X$ embeds into $\K_1^Y$ via an embedding that preserves 
undefined, then $X'' \leq_T Y$.
\end{lemma}

\begin{proof}
Let $F$ be an embedding of $\K_2^X$ into $\K_1^Y$ that preserves undefined.  Then $X' \leq_T 
Y$ by~\cite[Theorem~6.9]{GolovTerwijn}.  We use the embedding to show that 
$\mathrm{TOT}^X 
= \{e : \text{$\Phi_e^X$ is total}\}$ is c.e.\ relative to $Y$.  We already know that the 
complement $\overline{\mathrm{TOT}^X}$ is c.e.\ relative to $Y$ because $X' \leq_T Y$.  Hence 
it follows that $X'' \equiv_T \mathrm{TOT}^X \leq_T Y$.

Let $e_0$ be an index so that $\Phi_{e_0}^{f \oplus g} \simeq \Phi_{g(0)}^{g^-}$ for all $f$ and 
$g$, where $g^-(n) = g(n+1)$.  Let $e_1$ be an index so that
\begin{align*}
\Phi_{e_1}^{f \oplus g}(n) =
\begin{cases}
g(0)+1 & \text{if $n = 0$}\\
g(n) & \text{if $n > 0$}
\end{cases}
\end{align*}
for all $f$, $g$, and $n$.

Let $p = {e_0} \concat 0^\omega$, $q = {e_1} \concat 0^\omega$, and $g = 0 \concat X$.  These 
are all members of $\K_2^X$.  Observe that, for example, $q(q(qg))) = 3 \concat X$.  Write 
$q^{\bar{n}}g$ to denote $q(q \cdots (qg))$, with $n$ right-to-left applications of $q$.  Then 
$q^{\bar{n}}g = n \concat X$.  Thus
\begin{align*}
p \cdot (q^{\bar{n}}g) =
\begin{cases}
\Phi_n^X & \text{if $\Phi_n^X$ is total}\\
\uarrow & \text{otherwise}.
\end{cases}
\end{align*}
Let $a = F(p)$, $b = F(q)$, and $c = F(g)$.  Then
\begin{align*}
b^{\bar{n}}c = F(q)^{\bar{n}}F(g) = F(q^{\bar{n}}g) = F(n \concat X).
\end{align*}
In particular, $b^{\bar{n}}c\darrow$.  Thus
\begin{align*}
a \cdot (b^{\bar{n}}c) = F(p)F(n \concat X) =
\begin{cases}
F(p \cdot n \concat X) = F(\Phi_n^X) & \text{if $\Phi_n^X$ is total}\\
\uarrow & \text{otherwise}.
\end{cases}
\end{align*}
because $F$ preserves undefined.  Hence we have shown that $a \cdot (b^{\bar{n}}c)\darrow$ if 
and only if $\Phi_n^X$ is total.  That is, $\mathrm{TOT}^X = \{n : \K_1^Y \models a \cdot 
(b^{\bar{n}}c)\darrow \}$, which is c.e.\ relative to $Y$.  Thus $\mathrm{TOT}^X$ is c.e.\ relative 
to $Y$, so $X'' \leq_T Y$ as explained above.
\end{proof}

\begin{theorem}
For oracles $X$ and $Y$, $\K_2^X$ embeds into $\K_1^Y$ via an embedding that preserves 
undefined if and only if $X'' \leq_T Y$.
\end{theorem}

\begin{proof}
Suppose that $X'' \leq_T Y$.  Then
\begin{align*}
\K_2^X \emb \K_1^{X''} \emb \K_1^Y
\end{align*}
by embeddings that preserve undefined by Lemma~\ref{lem:K2embK1presundef} and the 
assumption $X'' \leq_T Y$.

Conversely, if $\K_2^X \emb \K_1^Y$ by an embedding that preserves undefined, then $X'' 
\leq_T Y$ by Lemma~\ref{lem:presundefjj}.
\end{proof}

Now we embed $\K_2^X$ into $K_1^{X'}$ (by an embedding that cannot preserve undefined) by 
first embedding $\K_2^X$ into a binary version that we call $\K_{2, 01}^X$ and then embedding 
$\K_{2, 01}^X$ into $\K_1^{X'}$.  First we define $\K_{2, 01}^X$ and show that it is a pca that 
embeds $\K_2^X$.

\begin{definition}
Let $\K_{2,01}$ be the partial applicative system on $2^\omega$ with application
\begin{align*}
A \cdot B =
\begin{cases}
\uarrow & \text{if $A = 0^\omega$}\\
\Phi_e^{A \oplus B} & \text{if $e$ is least with $A(e) = 1$}.
\end{cases}
\end{align*}
Here we assume that machines are $\{0,1\}$-valued, say by taking outputs $\mod 2$ as 
necessary.

For an oracle $X$, $\K_{2,01}^X$ is $\K_{2,01}$ restricted to binary sequences $A \leq_T X$.
\end{definition}

\begin{proposition}
$\K_{2,01}$ is a pca.  For every oracle $X$, $\K_{2, 01}^X$ is a pca.
\end{proposition}

\begin{proof}
We prove that $\K_{2,01}$ is a pca.  The proof that each $\K_{2,01}^X$ is a pca is the same.  
Indeed, we exhibit the combinators $k$ and $s$ in $\K_{2, 01}$ similar to how one exhibits 
these combinators in the $f \cdot g = \Phi_{f(0)}^{f \oplus g}$ encoding of $\K_2$.  These 
combinators are computable (they both have the form $0^{e-1}1 \concat 0^\omega$), so they 
are in every $\K_{2,01}^X$.

Let $e_0$ be an index so that for all $n \in \omega$ and $X, Y \in 2^\omega$, 
$\Phi_{e_0}^{(0^n1 
\concat X) \oplus Y} = X$.  Let $e_1$ be an index so that for all $A, X \in 2^\omega$, 
$\Phi_{e_1}^{A \oplus X} = 0^{e_0-1}1 \concat X$.  Let $k = 0^{e_1-1}1 \concat 0^\omega$.  
Then for any $X, Y \in 2^\omega$
\begin{align*}
k \cdot X = (0^{e_1-1}1 \concat 0^\omega) \cdot X = \Phi_{e_1}^{(0^{e_1-1}1 \concat 0^\omega) 
\oplus X} = 0^{e_0-1}1 \concat X
\end{align*}
and
\begin{align*}
k \cdot X \cdot Y = (0^{e_0-1}1 \concat X) \cdot Y = \Phi_{e_0}^{(0^{e_0-1}1 \concat X) \oplus Y} 
= X.
\end{align*}
Thus $kX\darrow$ and $kXY = X$.

Now let $e_0$ be an index so that for all $n \in \omega$ and $X, Y, Z \in 2^\omega$, 
$\Phi_{e_0}^{(0^n1 
\concat (X \oplus Y)) \oplus Z} \simeq XZ(YZ)$.  Let $e_1$ be an index so 
that for all $n \in \omega$ and $X, Y \in 2^\omega$, $\Phi_{e_1}^{(0^n1 \concat X) \oplus Y} = 
0^{e_0-1}1 \concat (X \oplus Y)$.  Let $e_2$ be an index so that for all $A, X \in 2^\omega$, 
$\Phi_{e_2}^{A 
\oplus X} = 0^{e_1-1}1 \concat X$.  Let $s = 0^{e_2-1}1 \concat 0^\omega$.  Then 
for any $X, Y, Z \in 2^\omega$
\begin{align*}
s \cdot X = (0^{e_2-1}1 \concat 0^\omega) \cdot X = \Phi_{e_2}^{(0^{e_2-1}1 \concat 0^\omega) 
\oplus X} = 0^{e_1-1}1 \concat X
\end{align*}
and
\begin{align*}
s \cdot X \cdot Y = (0^{e_1-1}1 \concat X) \cdot Y = \Phi_{e_1}^{(0^{e_1-1}1 \concat X) \oplus Y} 
= 0^{e_0-1}1 \concat(X \oplus Y)
\end{align*}
and
\begin{align*}
s \cdot X \cdot Y \cdot Z = (0^{e_0-1}1 \concat(X \oplus Y)) \cdot Z = \Phi_{e_0}^{(0^{e_0-1}1 
\concat(X \oplus Y)) \oplus Z} \simeq XZ(YZ)
\end{align*}
Thus $sXY \darrow$ and $sXYZ \simeq XZ(YZ)$.  So $\K_{2,01}$ is a pca.
\end{proof}

\begin{proposition} \label{prop:K2embK201}
$\K_2$ embeds into $\K_{2, 01}$ via an embedding that preserves undefined.  Likewise, for 
every $X$, $\K_2^X$ embeds into $\K_{2, 01}^X$ via an embedding that preserves undefined.
\end{proposition}

\begin{proof}
We show that $\K_2$ embeds into $\K_{2, 01}$ preserving undefined.  The relativized version is 
the same.

For a partial function $\varphi \colon \omega \to \omega$, let $\graph(\varphi) = \{\la x, y \ra : x 
\in \dom(\varphi) \land \varphi(x) = y\}$.  View any $X \in 2^\omega$ as the partial function 
$\psi_X 
\colon \omega \to \omega$ via
\begin{align*}
\psi_X(n) =
\begin{cases}
\text{the least $y$ such that $\la n, y \ra \in X$} & \text{if there is such a $y$}\\
\uarrow & \text{otherwise}.
\end{cases}
\end{align*}
Observe that $\psi_{\graph(\varphi)} = \varphi$ for any partial function $\varphi$.

Using the recursion theorem, let $e$ be an index so that for all $X, Y \in 2^\omega$
\begin{align*}
&\Phi_e^{(0^{e-1}1 \concat X) \oplus (0^{e-1}1 \concat Y)} =\\
&\quad\quad
\begin{cases}
0^{e-1}1 \concat \graph\left(\Phi_{\psi_X(0)}^{\psi_X \oplus \psi_Y}\right) & \text{if $\psi_X$, 
$\psi_Y$, 
and $\Phi_{\psi_X(0)}^{\psi_X \oplus \psi_Y}$ are total}\\
\uarrow & \text{otherwise}.
\end{cases}
\end{align*}

Let $F(f) = 0^{e-1}1 \concat \graph(f)$.  This embeds $\K_2$ into $\K_{2, 01}$ and preserves 
undefined.  Suppose that $f \cdot g\darrow = \Phi_{f(0)}^{f \oplus g} = h$.  Then
\begin{align*}
F(f) \cdot F(g) &= (0^{e-1}1 \concat \graph(f)) \cdot (0^{e-1}1 \concat \graph(g))\\
&= \Phi_e^{(0^{e-1}1 \concat \graph(f)) \oplus (0^{e-1}1 \concat \graph(g))}\\
&= 0^{e-1}1 \concat \graph\left(\Phi_{\psi_{\graph(f)}(0)}^{\psi_{\graph(f)} \oplus 
\psi_{\graph(g)}}\right)\\
&= 
0^{e-1}1 \concat \graph\left(\Phi_{f(0)}^{f \oplus g}\right)\\
&= 0^{e-1}1 \concat \graph(h)\\
&= F(h).
\end{align*}
Likewise, if $f \cdot g = \Phi_{f(0)}^{f \oplus g}\uarrow$, then 
$\Phi_{\psi_{\graph(f)}(0)}^{\psi_{\graph(f)} 
\oplus \psi_{\graph(g)}} = \Phi_{f(0)}^{f \oplus g}\uarrow$, so $F(f) \cdot 
F(g)\uarrow$.
\end{proof}

Now 
we embed $\K_{2,01}^X$ into $\K_1^{X'}$ for a given oracle $X$.  For this, recall that a 
\emph{Scott set} is a collection $\Sc \subseteq 2^\omega$ that is closed under Turing 
reducibility, Turing join, and such that every infinite tree $T \subseteq 2^{<\omega}$ in $\Sc$ has 
a path in $\Sc$.  That is, the Scott sets are the second-order parts of $\omega$-models of weak 
K\"{o}nig's lemma.  A Scott set $\Sc$ is a Turing ideal and therefore a sub-pca of $\K_{2,01}$.  If 
$X \in \Sc$, then $\K_{2,01}^X$ is a sub-pca of $\Sc$.

If $Y$ has $\pa$ degree relative to $X$, then $Y$ computes a set $Z$ such that $\Sc = \{Z^{[i]} : 
i \in \omega\}$ is a Scott set containing $X$, where $Z^{[i]} = \{n : \la i, n \ra \in Z\}$ denotes the 
$i$\textsuperscript{th} column of $Z$ (\cite{ScottPA}, see also \cite[Theorem 
4.11]{SimpsonPi01}). 
 It thus follows from the low basis theorem that there is a set $Z$ with $Z' 
\leq_T X'$ such that $\{Z^{[i]} : i \in \omega\}$ is a Scott set containing $X$.

\begin{lemma} \label{lem:K201embK1}
For every oracle $X$, $\K_{2,01}^X$ embeds into $\K_1^{X'}$.
\end{lemma}

\begin{proof}
As discussed above, let $Z$ be such that $Z' \leq_T X'$ and $\Sc = \{Z^{[i]} : i \in \omega\}$ is a 
Scott set containing $X$.  We embed $\Sc$ into $\K_1^{X'}$, thereby embedding $\K_{2,01}^X$ 
into $\K_1^{X'}$.

Scott sets are closed with respect to containing total $\{0,1\}$-valued extensions of partial 
$\{0,1\}$-valued 
functions.  This means that for every $e$ and $i$, there is a $j$ such that 
$Z^{[j]}$ 
is a total extension of $\Phi_e^{Z^{[i]}}$:
\begin{align*}
\forall n(\Phi_e^{Z^{[i]}}(n)\darrow \;\Longrightarrow\; \Phi_e^{Z^{[i]}}(n) = Z^{[j]}(n))
\end{align*}
Of course, if $\Phi_e^{Z^{[i]}}$ is total, then $Z^{[j]} = \Phi_e^{Z^{[i]}}$.  Notice that ``$Z^{[j]}$ is a 
total extension of $\Phi_e^{Z^{[i]}}$'' is a $\Pi_1$ property of $e$, $i$, and $j$ relative to $Z$.  
Thus, given $e$ and $i$, $X' \geq_T Z'$ can find the least $j$ such that $Z^{[j]}$ is a total 
extension of $\Phi_e^{Z^{[i]}}$.  Furthermore, for any $a$ and $b$, $Z^{[a]} \oplus Z^{[b]}$ is in 
$\Sc$ 
and therefore some $Z^{[c]}$ is a total extension of $Z^{[a]} \cdot Z^{[b]}$.  Thus the 
function
\begin{align*}
g(a,b) = \text{the least $c$ such that $Z^{[c]}$ is a total extension of $Z^{[a]} \cdot Z^{[b]}$}
\end{align*}
is computable from $X'$.  Note that if $(Z^{[a]} \cdot Z^{[b]})\darrow$, then $Z^{[g(a,b)]} = Z^{[a]} 
\cdot Z^{[b]}$.

From here on, the proof is similar to that of Lemma~\ref{lem:K2embK1presundef}.  Using the 
recursion theorem, let $e$ be an index such that $\Phi_e^{X'}(a)$ produces another index $p = 
p(a) > 0$ for a machine $\Phi_{p(a)}^{X'}$ behaving as follows.
\begin{itemize}
\item $\Phi_{p(a)}^{X'}(0) = a$.
\item For $n > 0$, $\Phi_{p(a)}^{X'}(n)$:
\begin{itemize}
\item First runs $\Phi_n^{X'}(0)$.  If it halts, lets $b = \Phi_n^{X'}(0)$.
\item Returns $\Phi_e^{X'}(g(a,b))$.
\end{itemize}
\end{itemize}

Note that for any $a$, $ea0 = p(a) \cdot 0 = \Phi_{p(a)}^{X'}(0) = a$.  Thus if $a \neq b$, then 
$ea \neq eb$ because $e a 0 = a \neq b = e b 0$.

Say that $a$ is the \emph{least representative} of $Z^{[a]}$ if there is no $a' < a$ with $Z^{[a']} = 
Z^{[a]}$.  We show that the mapping $Z^{[a]} \mapsto ea$ on the least representatives of the 
elements of $\Sc$ embeds $\Sc$ into $\K_1^{X'}$.  This mapping is injective because if $Z^{[a]} 
\neq Z^{[b]}$, then $a \neq b$, so $ea \neq eb$.

Suppose that $a$, $b$, and $c$ are the least representatives of $Z^{[a]}$, $Z^{[b]}$, and 
$Z^{[c]}$ and that $(Z^{[a]} \cdot Z^{[b]})\darrow = Z^{[c]}$.  Then $(ea)\cdot(eb) = ec$ in 
$\K_1^{X'}$:
\begin{align*}
(ea) 
\cdot (eb) = \Phi_{ea}^{X'}(eb) = \Phi_{p(a)}^{X'}(eb) = \Phi_e^{X'}(g(a,b)) = \Phi_e^{X'}(c) = 
ec.
\end{align*}
The third equality is because $eb0 = b$ as discussed above, so $\Phi_{p(a)}^{X'}(eb) = 
\Phi_e^{X'}(g(a,b))$.  The fourth equality is because $g(a,b)=c$ since $Z^{[c]}$ is the least 
representative of $Z^{[a]} \cdot Z^{[b]}$.
\end{proof}

\begin{theorem}
For oracles $X$ and $Y$, $\K_2^X$ embeds into $\K_1^Y$ if and only if $X' \leq_T Y$.
\end{theorem}

\begin{proof}
Suppose that $X' \leq_T Y$.  Then
\begin{align*}
\K_2^X \emb \K_{2,01}^X \emb \K_1^{X'} \emb \K_1^Y.
\end{align*}
The first embedding is by Proposition~\ref{prop:K2embK201}.  The second embedding is by 
Lemma~\ref{lem:K201embK1}.  The third embedding is because $X' \leq_T Y$.

Conversely, if $\K_2^X \emb \K_1^Y$, then $X' \leq_T Y$ by~\cite[Theorem~6.9]{GolovTerwijn}.
\end{proof}

\section*{Acknowledgments}

We thank Michael Rathjen, Shuwei Wang, Andrew Brooke-Taylor, 
and Asaf Karagila for  discussions about pcas and 
helpful comments about related set theoretic problems.

\bibliographystyle{amsplain}
\bibliography{nonembeddings}

@article{Engeler,
	author = {Engeler, Erwin},
	doi = {10.1007/BF02483849},
	fjournal = {Algebra Universalis},
	issn = {0002-5240,1420-8911},
	journal = {Algebra Universalis},
	mrclass = {03B40 (08A99)},
	mrnumber = {631731},
	mrreviewer = {Henk\ Barendregt},
	number = {3},
	pages = {389--392},
	title = {Algebras and combinators},
	url = {https://doi.org/10.1007/BF02483849},
	volume = {13},
	year = {1981}}

@book{TroelstravanDalenII,
	author = {Troelstra, Anne S. and van Dalen, Dirk},
	isbn = {0-444-70358-6},
	mrclass = {03-01 (03B15 03Fxx 03G30)},
	mrnumber = {970277},
	mrreviewer = {G.\ E.\ Mints},
	pages = {i--xviii and 345--880 and I--LII},
	publisher = {North-Holland Publishing Co., Amsterdam},
	series = {Studies in Logic and the Foundations of Mathematics},
	title = {{C}onstructivism in {M}athematics. {V}ol. {II}},
	volume = {123},
	year = {1988}}

@article{Terwijn2025,
	author = {Terwijn, Sebastiaan A.},
	doi = {10.46298/lmcs-21(2:17)2025},
	fjournal = {Log. Methods Comput. Sci.},
	issn = {1860-5974},
	journal = {Logical Methods in Computer Science},
	mrclass = {03B40},
	mrnumber = {4917049},
	mrreviewer = {Mariangiola\ Dezani-Ciancaglini},
	number = {2},
	pages = {Paper No. 17, 10},
	title = {Completions of {K}leene's second model},
	url = {https://doi.org/10.46298/lmcs-21(2:17)2025},
	volume = {21},
	year = {2025}}

@phdthesis{Swan,
	author = {Swan, Andrew W.},
	school = {University of Leeds},
	title = {{A}utomorphisms of {P}artial {C}ombinatory {A}lgebras and {R}ealizability {M}odels of {C}onstructive {S}et {T}heory},
	year = {2012}}

@incollection{SimpsonPi01,
	author = {Simpson, Stephen G.},
	booktitle = {{R}everse {M}athematics 2001},
	isbn = {1-56881-263-9},
	mrclass = {03F35 (03B30 03D30 03D80)},
	mrnumber = {2185446},
	pages = {352--378},
	publisher = {Assoc. Symbol. Logic, La Jolla, CA},
	series = {Lect. Notes Log.},
	title = {{$\Pi^0_1$} sets and models of {$\rm WKL_0$}},
	volume = {21},
	year = {2005}}

@article{ShaferTerwijn,
	author = {Shafer, Paul and Terwijn, Sebastiaan A.},
	doi = {10.1017/jsl.2021.50},
	fjournal = {J. Symb. Log.},
	issn = {0022-4812,1943-5886},
	journal = {The Journal of Symbolic Logic},
	mrclass = {03D28 (03B40 03D45)},
	mrnumber = {4347572},
	mrreviewer = {Robert\ S.\ Lubarsky},
	number = {3},
	pages = {1154--1188},
	title = {Ordinal analysis of partial combinatory algebras},
	url = {https://doi.org/10.1017/jsl.2021.50},
	volume = {86},
	year = {2021}}

@inproceedings{Scott,
	author = {Scott, Dana},
	booktitle = {Proceedings of the {T}hird {S}candinavian {L}ogic {S}ymposium ({U}niv. {U}ppsala, {U}ppsala, 1973)},
	mrclass = {02C20 (02F25 02F30)},
	mrnumber = {384495},
	mrreviewer = {A.\ S.\ Kuzichev},
	pages = {154--193},
	publisher = {North-Holland, Amsterdam-Oxford},
	series = {Stud. Logic Found. Math.},
	title = {Lamba calculus and recursion theory},
	volume = {Vol. 82},
	year = {1975}}

@incollection{ScottPA,
	author = {Scott, Dana},
	booktitle = {Proc. {S}ympos. {P}ure {M}ath., {V}ol. {V}},
	mrclass = {02.72},
	mrnumber = {141595},
	mrreviewer = {H.\ Ribeiro},
	pages = {117--121},
	publisher = {Amer. Math. Soc., Providence, RI},
	title = {Algebras of sets binumerable in complete extensions of arithmetic},
	year = {1962}}

@incollection{Rathjen,
	author = {Rathjen, Michael},
	booktitle = {Logic {C}olloquium '03},
	isbn = {978-1-56881-293-9; 1-56881-293-0},
	mrclass = {03F50 (03F35)},
	mrnumber = {2207359},
	mrreviewer = {Mariko\ Yasugi},
	pages = {282--314},
	publisher = {Assoc. Symbol. Logic, La Jolla, CA},
	series = {Lect. Notes Log.},
	title = {Realizability for constructive {Z}ermelo-{F}raenkel set theory},
	volume = {24},
	year = {2006}}

@book{vanOosten,
	author = {van Oosten, Jaap},
	isbn = {978-0-444-51584-1},
	mrclass = {03F60 (03-02 03G30 26E40)},
	mrnumber = {2479466},
	mrreviewer = {Colin\ McLarty},
	pages = {xvi+310},
	publisher = {Elsevier B. V., Amsterdam},
	series = {Studies in Logic and the Foundations of Mathematics},
	title = {{R}ealizability: {A}n {I}ntroduction to its {C}ategorical {S}ide},
	volume = {152},
	year = {2008}}

@incollection{vanOosten1999,
	author = {van Oosten, Jaap},
	booktitle = {{M}odels and {C}omputability ({L}eeds, 1997)},
	doi = {10.1017/CBO9780511565670.019},
	isbn = {0-521-63550-0},
	mrclass = {03B40 (03B15)},
	mrnumber = {1721178},
	mrreviewer = {Ugo\ de'Liguoro},
	pages = {389--405},
	publisher = {Cambridge Univ. Press, Cambridge},
	series = {London Math. Soc. Lecture Note Ser.},
	title = {A combinatory algebra for sequential functionals of finite type},
	url = {https://doi.org/10.1017/CBO9780511565670.019},
	volume = {259},
	year = {1999}}

@book{Odifreddi,
	author = {Odifreddi, Piergiorgio},
	isbn = {0-444-87295-7},
	mrclass = {03Dxx (03-02 03E15 03E45 03F30 68Q05)},
	mrnumber = {982269},
	mrreviewer = {Rodney\ G.\ Downey},
	pages = {xviii+668},
	publisher = {North-Holland Publishing Co., Amsterdam},
	series = {Studies in Logic and the Foundations of Mathematics},
	title = {{C}lassical {R}ecursion {T}heory},
	volume = {125},
	year = {1989}}

@book{LongleyNormann,
	author = {Longley, John and Normann, Dag},
	doi = {10.1007/978-3-662-47992-6},
	isbn = {978-3-662-47991-9; 978-3-662-47992-6},
	mrclass = {03-02 (03B40 03D65 68Q55)},
	mrnumber = {3525626},
	mrreviewer = {Robert\ S.\ Lubarsky},
	pages = {xvi+571},
	publisher = {Springer, Heidelberg},
	series = {Theory and Applications of Computability},
	title = {{H}igher-{O}rder {C}omputability},
	url = {https://doi.org/10.1007/978-3-662-47992-6},
	year = {2015}}

@phdthesis{Bethke,
	author = {Bethke, Ingemarie},
	school = {Universiteit van Amsterdam},
	title = {{N}otes on {P}artial {C}ombinatory {A}lgebras},
	year = {1988}}

@article{Klop,
	author = {Klop, Jan Willem},
	journal = {Bulletin of the European Association for Theoretical Computer Science},
	pages = {472--482},
	title = {Extending partial combinatory algebras},
	volume = {16},
	year = {1982}}

@book{KleeneVesley,
	author = {Kleene, Stephen Cole and Vesley, Richard Eugene},
	mrclass = {02.23},
	mrnumber = {176922},
	mrreviewer = {G.\ Kreisel},
	pages = {viii+206},
	publisher = {North-Holland Publishing Co., Amsterdam},
	title = {{T}he {F}oundations of {I}ntuitionistic {M}athematics, {E}specially in {R}elation to {R}ecursive {F}unctions},
	year = {1965}}

@article{GolovTerwijn,
	author = {Golov, Anton and Terwijn, Sebastiaan A.},
	doi = {10.1215/00294527-2023-0002},
	fjournal = {Notre Dame J. Form. Log.},
	issn = {0029-4527,1939-0726},
	journal = {Notre Dame Journal of Formal Logic},
	mrclass = {03B40 (03D28 03D80)},
	mrnumber = {4564838},
	number = {1},
	pages = {129--158},
	title = {Embeddings between partial combinatory algebras},
	url = {https://doi.org/10.1215/00294527-2023-0002},
	volume = {64},
	year = {2023}}

@article{FrittaionRathjen,
	author = {Frittaion, Emanuele and Rathjen, Michael},
	doi = {10.1093/logcom/exaa087},
	fjournal = {J. Logic Comput.},
	issn = {0955-792X,1465-363X},
	journal = {Journal of Logic and Computation},
	mrclass = {03E70 (03F10)},
	mrnumber = {4226219},
	mrreviewer = {Robert\ S.\ Lubarsky},
	number = {2},
	pages = {630--653},
	title = {Extensional realizability for intuitionistic set theory},
	url = {https://doi.org/10.1093/logcom/exaa087},
	volume = {31},
	year = {2021}}

@incollection{Feferman,
	author = {Feferman, Solomon},
	booktitle = {{A}lgebra and {L}ogic ({F}ourteenth {S}ummer {R}es. {I}nst., {A}ustral. {M}ath. {S}oc., {M}onash {U}niv., {C}layton, 1974)},
	mrclass = {02E05 (02D99)},
	mrnumber = {409137},
	mrreviewer = {G.\ E.\ Mints},
	pages = {87--139},
	publisher = {Springer, Berlin-New York},
	series = {Lecture Notes in Math.},
	title = {A language and axioms for explicit mathematics},
	volume = {Vol. 450},
	year = {1975}}

@phdthesis{Longley,
	author = {Longley, John R.},
	school = {University of Edinburgh},
	title = {{R}ealizability {T}oposes and {L}anguage {S}emantics},
	year = {1994}}

@book{Beeson,
	author = {Beeson, Michael J.},
	doi = {10.1007/978-3-642-68952-9},
	isbn = {3-540-12173-0},
	mrclass = {03F65 (00A25 03F50 03F55 03F60)},
	mrnumber = {786465},
	mrreviewer = {V.\ Ya.\ Kreinovich},
	note = {Metamathematical studies},
	pages = {xxiii+466},
	publisher = {Springer-Verlag, Berlin},
	series = {Ergebnisse der Mathematik und ihrer Grenzgebiete (3) [Results in Mathematics and Related Areas (3)]},
	title = {{F}oundations of {C}onstructive {M}athematics},
	url = {https://doi.org/10.1007/978-3-642-68952-9},
	volume = {6},
	year = {1985}}

@book{Barendregt,
	author = {Barendregt, Henk P.},
	edition = {revised},
	isbn = {978-1-84890-066-0},
	mrclass = {03B40 (01A75 03-02)},
	mrnumber = {3235567},
	note = {With addenda for the 6th imprinting, Mathematical Logic and Foundations},
	pages = {xvi+622+E16},
	publisher = {College Publications, London},
	series = {Studies in Logic (London)},
	title = {{T}he {L}ambda {C}alculus: {I}ts {S}yntax and {S}emantics},
	volume = {40},
	year = {2012}}

@article{AbrahamShore,
	author = {Abraham, Uri and Shore, Richard A.},
	doi = {10.1007/BF02772668},
	fjournal = {Israel J. Math.},
	issn = {0021-2172},
	journal = {Israel Journal of Mathematics},
	mrclass = {03D30 (06A12)},
	mrnumber = {861896},
	mrreviewer = {C.\ G.\ Jockusch, Jr.},
	number = {1},
	pages = {1--51},
	title = {Initial segments of the degrees of size {$\aleph_1$}},
	url = {https://doi.org/10.1007/BF02772668},
	volume = {53},
	year = {1986}}

\end{document}